\documentclass[11pt,twoside]{article}
\usepackage[utf8]{inputenc}
\usepackage{amsmath}
\usepackage{amsfonts}
\usepackage{amssymb}
\usepackage{amsthm}
\usepackage[english]{babel}
\usepackage{hyperref}
\usepackage[alphabetic]{amsrefs}
\usepackage{geometry}
\usepackage{graphicx}
\usepackage{mathrsfs}
\usepackage{tikz-cd}
\usepackage[indent]{parskip}
\usepackage{paralist}

\hypersetup{
	colorlinks   = true, 
	urlcolor     = blue, 
	linkcolor    = blue,
	citecolor    = blue 
}

\usepackage{fancyhdr}
\usepackage{mathtools}
\mathtoolsset{showonlyrefs}

\makeatletter
\def\thm@space@setup{%
	\thm@preskip=\parskip \thm@postskip=0pt
}
\makeatother
\usepackage{etoolbox}
\AtBeginEnvironment{proof}{\vspace*{-0.5\parskip}}

\begin{document}
\title{\textbf{Existence of closed geodesics on certain non-compact Riemannian manifolds}}
\author{Akashdeep Dey \thanks{Department of Mathematics, University of Toronto, Toronto, Ontario M5S 2E4, Canada, Email:  dey.akash01@gmail.com}}
\date{}
\maketitle

\theoremstyle{plain}
\newtheorem{thm}{Theorem}[section]
\newtheorem{lem}[thm]{Lemma}
\newtheorem{pro}[thm]{Proposition}
\newtheorem{clm}[thm]{Claim}
\newtheorem{cor}[thm]{Corollary}
\newtheorem*{thm*}{Theorem}
\newtheorem*{lem*}{Lemma}
\newtheorem*{clm*}{Claim}

\theoremstyle{definition}
\newtheorem{defn}[thm]{Definition}
\newtheorem{ex}[thm]{Example}
\newtheorem{rmk}[thm]{Remark}

\numberwithin{equation}{section}

\newcommand{\mf}{manifold\;}\newcommand{\mfs}{manifolds\;}
\newcommand{\vf}{varifold\;}
\newcommand{\hy}{hypersurface\;}
\newcommand{\Rm}{Riemannian\;}
\newcommand{\cn}{constant\;}
\newcommand{\mt}{metric\;} 
\newcommand{\st}{such that\;}\newcommand{\sot}{so that\;}
\newcommand{\Thm}{Theorem\;}
\newcommand{\Lem}{Lemma\;}
\newcommand{\Pro}{Proposition\;}
\newcommand{\eqn}{equation\;}
\newcommand{\te}{there exist\;}\newcommand{\tes}{there exists\;}\newcommand{\Te}{There exist\;}\newcommand{\Tes}{There exists\;}\newcommand{\fa}{for all\;}
\newcommand{\tf}{Therefore,\;} \newcommand{\hn}{Hence,\;}\newcommand{\Sn}{Since\;}\newcommand{\sn}{since\;}\newcommand{\nx}{\Next,\;}\newcommand{\df}{define\;}
\newcommand{\wrt}{with respect to\;}
\newcommand{\bbr}{\mathbb{R}}\newcommand{\bbq}{\mathbb{Q}}
\newcommand{\bbn}{\mathbb{N}}
\newcommand{\bbz}{\mathbb{Z}}
\newcommand{\mres}{\scalebox{1.8}{$\llcorner$}}
\newcommand{\ra}{\rightarrow}
\newcommand{\fn}{function\;}
\newcommand{\lra}{\longrightarrow}
\newcommand{\sps}{Suppose\;}
\newcommand{\del}{\partial}
\newcommand{\seq}{sequence\;}\newcommand{\w}{with\;}
\newcommand{\cts}{continuous\;} 
\newcommand{\bX}{\mathbf{X}} 
\newcommand{\bA}{\mathbf{A}} 
\newcommand{\bfe}{\mathbf{e}}\newcommand{\bc}{\mathbf{c}}\newcommand{\bj}{\mathbf{j}}
\newcommand{\bL}{\mathbf{L}}\newcommand{\bH}{\mathbf{H}}
\newcommand{\cm}{\mathcal{C}(M)}
\newcommand{\zn}{\mathcal{Z}_n(M, \mathbb{Z}_2)}
\newcommand{\Le}{\sL^E}\newcommand{\msu}{M \setminus U}

\newcommand{\cA}{\mathcal{A}}\newcommand{\cB}{\mathcal{B}}\newcommand{\cC}{\mathcal{C}}\newcommand{\cD}{\mathcal{D}}\newcommand{\cE}{\mathcal{E}}\newcommand{\cF}{\mathcal{F}}\newcommand{\cG}{\mathcal{G}}\newcommand{\cH}{\mathcal{H}}\newcommand{\cI}{\mathcal{I}}\newcommand{\cJ}{\mathcal{J}}\newcommand{\cK}{\mathcal{K}}\newcommand{\cL}{\mathcal{L}}\newcommand{\cM}{\mathcal{M}}\newcommand{\cN}{\mathcal{N}}\newcommand{\cO}{\mathcal{O}}\newcommand{\cP}{\mathcal{P}}\newcommand{\cQ}{\mathcal{Q}}\newcommand{\cR}{\mathcal{R}}\newcommand{\cS}{\mathcal{S}}\newcommand{\cT}{\mathcal{T}}\newcommand{\cU}{\mathcal{U}}\newcommand{\cV}{\mathcal{V}}\newcommand{\cW}{\mathcal{W}}\newcommand{\cX}{\mathcal{X}}\newcommand{\cY}{\mathcal{Y}}\newcommand{\cZ}{\mathcal{Z}}

\newcommand{\sA}{\mathscr{A}}\newcommand{\sB}{\mathscr{B}}\newcommand{\sC}{\mathscr{C}}\newcommand{\sD}{\mathscr{D}}\newcommand{\sE}{\mathscr{E}}\newcommand{\sF}{\mathscr{F}}\newcommand{\sG}{\mathscr{G}}\newcommand{\sH}{\mathscr{H}}\newcommand{\sI}{\mathscr{I}}\newcommand{\sJ}{\mathscr{J}}\newcommand{\sK}{\mathscr{K}}\newcommand{\sL}{\mathscr{L}}\newcommand{\sM}{\mathscr{M}}\newcommand{\sN}{\mathscr{N}}\newcommand{\sO}{\mathscr{O}}\newcommand{\sP}{\mathscr{P}}\newcommand{\sQ}{\mathscr{Q}}\newcommand{\sR}{\mathscr{R}}\newcommand{\sS}{\mathscr{S}}\newcommand{\sT}{\mathscr{T}}\newcommand{\sU}{\mathscr{U}}\newcommand{\sV}{\mathscr{V}}\newcommand{\sW}{\mathscr{W}}\newcommand{\sX}{\mathscr{X}}\newcommand{\sY}{\mathscr{Y}}\newcommand{\sZ}{\mathcal{Z}}

\newcommand{\al}{\alpha}\newcommand{\be}{\beta}\newcommand{\ga}{\gamma}\newcommand{\de}{\delta}\newcommand{\ve}{\varepsilon}\newcommand{\et}{\eta}\newcommand{\ph}{\phi}\newcommand{\vp}{\varphi}\newcommand{\ps}{\psi}\newcommand{\ka}{\kappa}\newcommand{\la}{\lambda}\newcommand{\om}{\omega}\newcommand{\rh}{\rho}\newcommand{\si}{\sigma}\newcommand{\tht}{\theta}\newcommand{\ta}{\tau}\newcommand{\ch}{\chi}\newcommand{\ze}{\zeta}\newcommand{\Ga}{\Gamma}\newcommand{\De}{\Delta}\newcommand{\Ph}{\Phi}\newcommand{\Ps}{\Psi}\newcommand{\La}{\Lambda}\newcommand{\Om}{\Omega}\newcommand{\Si}{\Sigma}\newcommand{\Tht}{\Theta}\newcommand{\na}{\nabla}

\newcommand{\nm}[1]{\left\|#1\right\|}\newcommand{\md}[1]{\left|#1\right|}\newcommand{\Md}[1]{\Big|#1\Big|}\newcommand{\db}[1]{[\![#1]\!]}
\newcommand{\vol}{\operatorname{Vol}}\newcommand{\ov}[1]{\overline{#1}}

\vspace{-2ex}
\begin{abstract}
	\vspace{-1.5ex}
	\noindent
Let $M$ be a complete Riemannian manifold. Suppose $M$ contains a bounded, concave, connected open set $U$ with $C^0$ boundary and $M\setminus U$ is connected. We assume that either the relative homotopy set $\pi_1(M,M\setminus U)=0$ or the union of all the conjugate subgroups of the image of the homomorphism $\pi_1(M\setminus U)\rightarrow \pi_1(M)$ (induced by the inclusion $M\setminus U\hookrightarrow M$) is a proper subset of $\pi_1(M)$. (The first condition is equivalent to $\pi_1(M\setminus U)\rightarrow \pi_1(M)$ is surjective; the second condition is satisfied if the relative homology group $H_1(M,M\setminus U)\neq 0$.) Then there exists a non-trivial closed geodesic on $M$. This partially proves a conjecture of Chambers, Liokumovich, Nabutovsky and Rotman \cite{CLNR}*{Conjecture 1.1}.
\end{abstract}

\section{Introduction}

In \cite{Poi}, Poincar\'{e} asked the question whether every closed \Rm surface $\Si$ admits a closed geodesic. If the fundamental group of $\Si$ (more generally, any closed \Rm \mf $N$) is non-zero, then one can construct a non-trivial closed geodesic on \(\Si\) (resp. on $N$) by minimizing the length functional in a non-trivial (free) homotopy class of maps $S^1 \ra \Si$ (resp. $S^1 \ra N$). In \cite{Bir}, Birkhoff developed a min-max method and proved the existence of a closed geodesic on every \Rm two-sphere. This theorem was generalized in higher dimensions by Fet and Lyusternik \cite{LF}; they proved that every closed \Rm \mf admits a closed geodesic.

Minimal submanifolds are the higher dimensional generalizations of the geodesics. In \cite{alm}, Almgren developed a min-max theory to prove the existence of a closed (singular) minimal submanifold of dimension $l$ in any closed \Rm \mf \(N^n\) and for every \(1\leq l\leq n-1\). The regularity theory in the co-dimension 1 case (i.e. when $l=n-1$) was further developed by Pitts \cite{pit} and Schoen–Simon \cite{ss}. When \(l=1\), the Almgren–Pitts min–max theory produces stationary geodesic nets, which may not be smooth. (For a precise statement, see \cite{LS}*{Proposition 3.2}.) Using the phase-transition regularization of the length functional, Chodosh and Mantoulidis \cite{CM23} proved that the geodesic nets, produced by the Almgren-Pitts min-max theory on surfaces, are closed geodesics. In \cite{LS}, Liokumovich and Staffa proved that for a generic \Rm \mt on a closed manifold, the union of all stationary geodesic nets is dense. In \cite{LiS}, Li and Staffa proved the generic equidistribution of closed geodesics on closed surfaces and stationary geodesic nets on closed manifolds.

In general, the above mentioned Fet-Lyusternik theorem does not hold on non-compact manifolds. For example, \(\bbr^n\), equipped with a \Rm metric of non-positive curvature, does not contain a closed geodesic. By the works of Thorbergsson \cite{Tho} and Bangert \cite{Ban}, every complete \Rm surface with locally convex ends admits a closed geodesic. Informally speaking, a complete \Rm \mf $M$ is said to have locally convex ends if $M$ contains a bounded, concave (see Definition \ref{d.conv}), open set. In \cite{CLNR}, Chambers, Liokumovich, Nabutovsky and Rotman proved that every complete \Rm \mf with locally convex ends admits a stationary geodesic net, which has at most one singular vertex. They also conjectured that such a manifold contains a closed geodesic. In this paper, we prove this conjecture (Theorems \ref{thm}, \ref{thm.main}, Remark \ref{r.1.4}) under some additional topological assumptions.

\begin{defn}\label{d.conv}
Let \(\sM\) be a complete \Rm manifold. \(\sC\subset \sM\) is called \textit{\(\de\)-convex} (where $\de>0$) if the following condition is satisfied. If \(p,q\in \sC\) \sot \(d(p,q)\leq \de\), then a minimizing geodesic from $p$ to $q$ is contained in $\sC$. \(\sU \subset \sM\) is called \textit{concave} if \(\sM \setminus \sU\) is \(\de\)-convex for some $\de>0$.
\end{defn}

For a topological space \(X\) and $A\subset X$, the inclusion $A \hookrightarrow X$ induces a map $\pi_0(\La A)\ra \pi_0(\La X)$ (here $\pi_0(\cdot)$ denotes the set of path-components and $\La \cdot$ denotes the free loop space) and a homomorphism \(\pi_i(\La A, c)\ra \pi_i(\La X, c)\) (where $\pi_i(\cdot)$ denotes the $i$-th homotopy group) \fa \(i\geq 1\) and basepoints \(c\in \La A\). We say that $(X,A)$ satisfies condition $(*)$ if $X$, $A$, $X\setminus A$ are connected, $\pi_1(X,A)\neq 0$, the map $\pi_0(\La A)\ra \pi_0(\La X)$ is bijective and the homomorphism \(\pi_i(\La A, c)\ra \pi_i(\La X, c)\) is an isomorphism \fa \(i\geq 1\) and  \(c\in \La A\) which are non-contractible in $A$.

\begin{thm}\label{thm}
Let $M$ be a complete Riemannian manifold. Suppose $M$ contains a bounded, concave, connected, open set $U$ with $C^0$ boundary and $M\setminus U$ is connected. Then either \tes a non-trivial closed geodesic on $M$ or $(M,M\setminus U)$ satisfies condition $(*)$.
\end{thm} 

As a consequence of this theorem, we obtain the following.

\begin{thm}\label{thm.main}
Let $M$ be a complete Riemannian manifold. Suppose $M$ contains a bounded, concave, connected, open set $U$ with $C^0$ boundary and $M\setminus U$ is connected. Then there exists a non-trivial closed geodesic on $M$ if one of the following conditions is satisfied.
\begin{itemize}
	\item[(i)] The relative homotopy set \(\pi_1(M,M\setminus U)=0\).
	\item[(ii)] The relative homology group \(H_1(M,M\setminus U)\neq 0\).
	\item[(iii)] \(\del U\) is not connected.
	\item[(iv)] \(\del U\) is connected and simply connected.
	\item[(v)] $\pi_1(M)$ is a finite group or an Abelian group.
	\item[(vi)] $\pi_1(U)$ is a finite group or an Abelian group.
\end{itemize}
\end{thm}

\begin{rmk}\label{r.1.4}
Let $M$, $U$ be as in the assumption of Theorem \ref{thm}. Then, \tes a closed geodesic on $M$ if either \(\pi_1(M,\msu)=0\) or the map \(\pi_0(\La (\msu))\ra \pi_0(\La M)\) is not surjective. \Sn $\msu$ is connected, the homotopy exact sequence for the pair $(M,\msu)$ implies that \(\pi_1(M,\msu)=0\) if and only if $\pi_1(M\setminus U)\rightarrow \pi_1(M)$ is surjective. By the Seifert-van Kampen theorem, if $\del U$ is connected and $\pi_1(\del U)\ra \pi_1(\ov{U})$ is surjective, then $\pi_1(M\setminus U)\rightarrow \pi_1(M)$ is also surjective. In the proof of Theorem \ref{thm.main}, we show that if $\del U$ is not connected, then \(\pi_0(\La (\msu))\ra \pi_0(\La M)\) is not surjective. \tf using Lemma \ref{l.*}, we conclude that if $M$ does not admit a closed geodesic, the following condition must be satisfied. The image of $\pi_1(M\setminus U)\rightarrow \pi_1(M)$ is a proper subgroup of $\pi_1(M)$ and the union of all its conjugate subgroups is $\pi_1(M)$. This condition implies the following condition on $\ov{U}$ (which is a compact manifold with non-empty boundary). $\del U$ is connected, the image of $\pi_1(\del U)\rightarrow \pi_1(\ov{U})$ is a proper subgroup of $\pi_1(\ov{U})$ and the union of all its conjugate subgroups is $\pi_1(\ov{U})$.
\end{rmk}

Let us briefly discuss the proof of Theorem \ref{thm} and elaborate on the condition $(*)$. When $\pi_1(M,M\setminus U)=0$, the proof for the existence of a closed geodesic on $M$ is similar to the proof of the Fet-Lyusternik theorem \cite{Moore}*{Proof of Theorem 2.3.2}. If \(\pi_1(M,M\setminus U)=0\), we choose \(k\geq 1\) \st \(\pi_i(M,M\setminus U)=0\) for \(1\leq i \leq k\) and \(\pi_{k+1}(M,M\setminus U)\neq 0\). This implies \(\pi_k(\La M,\La (M\setminus U),c_p)\neq 0\) (here $c_p$ is the constant loop defined by $p$; see Section \ref{s.notn}). We choose \(F:(I^k,\del I^k)\ra (\La M,\La (M\setminus U)) \) \st \([F]\neq 0\) in \(\pi_k(\La M,\La (M\setminus U),c_p)\). We prove the existence of a closed geodesic on $M$ by applying the Birkhoff curve shortening flow to the map $F$. The Birkhoff curve shortening flow is a discrete gradient flow for the length functional and (for the appropriately chosen parameters) the flow maps $\La (M\setminus U)$ to itself (since $U$ is concave).  

If the map $\pi_0(\La (M\setminus U))\ra \pi_0(\La M)$ is not injective or if \(\pi_m(\La M, \La (M\setminus U), c^*)\neq 0\) for some \(m\geq 1\) and some \(c^* \in \La (M\setminus U)\) which is non-contractible in \(M\setminus U\), one can prove the existence of a closed geodesic on $M$ using a min-max argument (similar to the above case when \(\pi_1(M,M\setminus U)=0\)). 

On the other hand, when \(\pi_1(M,M\setminus U)\neq 0\), we aim to prove the existence of a closed geodesic on $M$ using a minimization method. However, on a non-compact manifold, it may not be possible to obtain a closed geodesic by minimization since under the Birkhoff curve shortening flow, a loop might escape to infinity. So, for our proof method to work, we need to additionally assume that \(M\setminus U\) is connected (so that a non-trivial element in $\pi_1(M,M\setminus U)$ can be represented by a loop in $M$) and the map $\pi_0(\La (M\setminus U))\ra \pi_0(\La M)$ is not surjective. The map $\pi_0(\La (M\setminus U))\ra \pi_0(\La M)$ fails to be surjective if and only if there exists a loop $\ga$ in $M$ which is not freely homotopic to a loop in $M\setminus U$. As a consequence, under the Birkhoff curve shortening flow, $\ga$ can not escape to infinity and hence must converge to a non-trivial closed geodesic. We remark that for the above mentioned min-max argument, we do not need to assume that $\msu$ is connected; however, if \(M\setminus U\) is not connected, $\pi_1(M,M\setminus U)$ is necessarily non-zero (cf. proof of Theorem \ref{thm} when either (C2) or (C3) holds in Section \ref{s.proof}).

The condition $(*)$ for $(M,M\setminus U)$ has formal similarities with the topology of the free loop space $\La N$, where $N$ is a closed \Rm \mf \w negative sectional curvature. $(M,M\setminus U)$ satisfies condition $(*)$ if and only if $\pi_1(M, M\setminus U)\neq 0$, for each $\tilde{\Pi}\in \pi_0(\La M)$ \tes a unique $\Pi\in \pi_0(\La (M\setminus U))$ \st \(\Pi\subset \tilde{\Pi}\) and for each non-standard path-component $\tilde{\Pi}\in \pi_0(\La M)$ (i.e. $\tilde{\Pi}$ does not contain the constant loops), $\pi_i(\tilde{\Pi}, \Pi)=0$ \fa $i\geq 1$. (We note that, if $M\setminus U$ is connected, then $\pi_1(M, M\setminus U)\neq 0$ implies $\pi_1(M)\neq 0$, i.e. $\La M$ has more than one path-components.) On the other hand, if $N$ is a closed \Rm \mf \w negative sectional curvature, then $\pi_1(N)\neq 0$; the standard path-component of $\La N$ is homotopy equivalent to $N$ and does not contain any non-trivial closed geodesic; each non-standard path-component $\Pi$ of $\La N$ is homotopy equivalent to $S^1$ and contains a unique (up to reparametrization) closed geodesic, which is the minimizer of the length functional in $\Pi$.

\noindent\textbf{Acknowledgements.} I am very grateful to  Professors Yevgeny Liokumovich, Alexander Nabutovsky and Regina Rotman for many helpful discussions and suggestions, especially about Theorem \ref{thm.main} and Remark \ref{r.1.4}. I would like to thank Prof. Liokumovich for suggesting this problem to me. The author is partially supported by NSERC Discovery grant.

\section{Notations and Preliminaries}
\subsection{Notations}\label{s.notn}
Here we summarize the notations which will be frequently used later.

\begin{itemize}
	\item $[m]$ : The set \(\{1,2,\dots,m\}\). 	
	\item $I$ : The interval $[0,1]$.
	\item \(\text{conv}_p(N)\) : The convexity radius of $N$ at $p$.
	\item \(\text{conv}(N)\) : The convexity radius of $N$.
	\item \(\text{inj}(N)\) : The injectivity radius of \(N\).
	\item \(B(p,r)\) : The geodesic ball centered at $p$ \w radius $r$.
	\item \(\ov{A}\) : The closure of \(A\) (in a topological space).
	\item We will identify \(S^1\) with the quotient of the unit interval \([0,1]\). In particular, a \cts map \(\ga:S^1 \ra X\) is same as a \cts map \(\ga: [0,1]\ra X\) \w \(\ga(0)=\ga(1)\).
	\item \(\La X\) : The space of \cts maps from \(S^1\) to \(X\) (where \(X\) is a topological space), endowed with the compact-open topology.
	\item \(\Om_pX = \{\ga\in \La X: \ga(0)=\ga(1)=p\}.\)
	\item \(c_p\) : For \(p\in X\) (where \(X\) is a topological space), $c_p\in \La X$ is the constant curve defined by $p$, i.e. \(c_p(s)=p\) \fa $s\in S^1$.
	\item \(\md{\ga}\) : For \(\ga\in \La X\) (where \(X\) is a topological spaces), \(\md{\ga}\) denotes the image of \(\ga\).
	\item \(\sL N\) : The Sobolev space of \(W^{1,2}\) maps from \(S^1\) to \(N\) (where \(N\) is a \Rm manifold).
	\item \(\ell(\ga)\), \(\sE(\ga)\) : For \(\ga \in \sL N\), \(\ell(\ga)\) and \(\sE(\ga)\) denote the length and energy of \(\ga\), i.e.
	\[\ell(\ga)=\int_{0}^{1}\md{\ga'},\quad \sE(\ga)=\int_{0}^{1}\md{\ga'}^2.\]
	
\end{itemize}

\subsection{The Birkhoff curve shortening map}
In this subsection, we briefly discuss the properties of the Birkhoff curve shortening map following Colding-Minicozzi \cite{CM_article,CM_book}. Let $N$ be a closed \Rm manifold. We fix \(E, R>0\) and \(L\in \bbn\) so that 
\begin{equation}\label{e.R}
0<R<\min\left\{\frac{1}{2}\text{conv}(N), \frac{1}{4}\text{inj}(N)\right\}\quad\text{and}\quad L\geq \max\left\{\frac{E}{R^2},\sqrt{E}\right\}.
\end{equation}
For fixed $(E,R,L)$ satisfying \eqref{e.R}, let \(\La\) denote the set of all maps \(c: S^1 \ra N\) \st the following conditions are satisfied. \Te \(L\) points \(s_0,s_1,s_2,\dots,s_L=s_0\in S^1\) \st for all \(i\in [L]\), $d(c(s_{i-1}),c(s_i))\leq R$; \(c\) restricted to \([s_{i-1},s_i]\) is the unique minimizing constant speed geodesic from \(c(s_{i-1})\) to \(c(s_{i})\) and \(\ell(c)\leq L\). Then, \(\La \subset \sL N\) is compact (\wrt the subspace topology). 

Let us use the notation \(\sL^E=\{c\in\sL N: \sE(c)\leq E\}\). Let \(\ga\in\sL^E\) and \(0\leq x<y\leq 1\) \w \((y-x)\leq 1/L\). Then, using \eqref{e.R},
\begin{equation}\label{e.sobolev}
d(\ga(x),\ga(y))\leq \int_{x}^{y}\md{\ga'}\leq \left(\int_{x}^{y} 1\right)^{1/2}\left(\int_{x}^{y}\md{\ga '}^2\right)^{1/2}\leq \sqrt{\frac{E}{L}}\leq R.
\end{equation}

\begin{defn}\label{d.Ps}

The maps \(\Ps_e,\Ps_o,\tilde{\Ps}:\sL^E\ra\sL^E\) are defined as follows (cf. \cite{CM_book}*{p. 166, Steps 1-3}). For \(i\in\{0,1,2,\dots,2L\}\), we set \(x_i=\frac{i}{2L}\).
\begin{itemize}
	\item Denoting \(\ga_e=\Ps_e(\ga)\), \(\ga_e\) is defined by the condition that for all \(i\in[L]\), \(\ga_e(x_{2i})=\ga(x_{2i})\) and \(\ga_e\) restricted to the interval \([x_{2i-2},x_{2i}]\) is a minimizing \cn speed geodesic. (By \eqref{e.sobolev} and \eqref{e.R}, \tes a unique minimizing geodesic from $\ga(x_{2i-2})$ to \(\ga(x_{2i})\).)
	\item Denoting \(\ga_o=\Ps_o(\ga)\), \(\ga_o\) is defined by the condition that for all \(i\in[L]\), \(\ga_o(x_{2i-1})=\ga(x_{2i-1})\) and \(\ga_o\) restricted to the interval \([x_{2i-1},x_{2i+1}]\) (here $x_{2L+1}=x_1$) is a minimizing \cn speed geodesic.
	\item Denoting \(\tilde{\ga}=\tilde{\Ps}(\ga)\), \(\tilde{\ga}\) is defined by the condition that \(\tilde{\ga}\) is a reparametrization of \(\ga\) \sot \(\tilde{\ga}(0)=\ga(0)\) and \(\tilde{\ga}\) has constant speed. 
\end{itemize}
The Birkhoff curve shortening map \(\Ps:\sL^E \ra \La\) is defined by $\Ps(\ga)=\tilde{\Ps}(\Ps_o(\Ps_e(\ga)))$.
	
\end{defn}

\begin{pro}\label{p.Ph}
	\Tes a \cts map \(\Ph_e:\Le\times[0,1]\ra \sL N\) \st \(\Ph_e(\ga,0)=\ga\), \(\Ph_e(\ga,1/2)=\Ps_e(\ga)\) and \(\Ph_e(\ga,1)=\tilde{\Ps}(\Ps_e(\ga)).\) Similarly, \tes  a \cts map \(\Ph_o:\Le\times[0,1]\ra \sL N\) \st \(\Ph_o(\ga,0)=\ga\), \(\Ph_o(\ga,1/2)=\Ps_o(\ga)\) and \(\Ph_o(\ga,1)=\tilde{\Ps}(\Ps_o(\ga))\). \tf \tes a \cts map \(\Ph:\Le\times[0,1]\ra \sL N\) \st \(\Ph(\ga,0)=\ga\) and \(\Ph(\ga,1)=\Ps(\ga)\).
\end{pro}

\begin{proof}
	The proof of this proposition follows from the proofs of Lemma 5.3 and Lemma 5.5 in \cite{CM_book}*{Chapter 5}. Let us briefly present the argument here. Let
	\[\Ga=\{(p,q)\in N\times N: d(p,q)<\text{conv}(N)\}.\]
	We define \(\bH:\Ga \ra C^1([0,1],N)\) as follows (\cite{CM_book}*{p. 170, item (C2)}). \(\bH(p,q):[0,1]\ra N\) is the minimizing \cn speed geodesic from $p=\bH(p,q)(0)$ to \(q=\bH(p,q)(1)\), i.e. 
\begin{equation}\label{e.bH}
\bH(p,q)(t)= \exp_p(t\exp_p^{-1}(q)).
\end{equation}	
 Let us use the notation \(\tilde{\ga}_{e}=\tilde{\Ps}(\Ps_e(\ga)).\) Then, \(\ga_e=\tilde{\ga}_e\circ P_{\ga}\) where \(P_{\ga}:[0,1]\ra [0,1]\) is a monotone reparametrization of \([0,1]\) that fixes \(0\) and \(1\). For $(\ga , s)\in \Le\times [0,1]$, \(\Ph_e(\ga,s):S^1\ra N\) is defined by (\cite{CM_book}*{p. 174, eq. (5.24) and (5.25)})
\begin{equation*}
\Ph_e(\ga, s)(x)=
\begin{cases}
\bH(\ga(x),\ga_e(x))(2s) &\text{ if } s\in [0,1/2];\\
\tilde{\ga}_e\left((2-2s)P_{\ga}(x)+(2s-1)x\right) &\text{ if } s\in [1/2,1].
\end{cases}
\end{equation*}
One can define \(\Ph_o\) in a similar way. Finally, \(\Ph:\Le\times[0,1]\ra \sL N\) is defined by 
\begin{equation*}
\Ph(\ga, s)=
\begin{cases}
\Ph_e(\ga , s) &\text{ if } s\in [0,1/2];\\
\Ph_o(\ga_e, 2s-1) &\text{ if } s\in [1/2,1].
\end{cases}
\end{equation*}
\end{proof}

\begin{rmk}\label{r.spt.Ps.Ph}
From Definition \ref{d.Ps} and Proposition \ref{p.Ph}, it follows that
\[\md{\Ps(\ga)},\md{\Ph(\ga,s)} \subset \cN_{4R}(\md{\ga})\; \forall\;\ga\in\Le \text{ and }s\in [0,1].\]
Moreover, if \(C\subset N\) is $2R$-convex (Definition \ref{d.conv}) and \(\md{\ga}\subset C\), then
\[\md{\Ps(\ga)},\md{\Ph(\ga,s)} \subset C\; \forall\;s\in [0,1].\]
\end{rmk}

The space $\sL N$ is metrizable. Let us fix a \mt \(\mathbf{d}\) on $\sL N$.  Let \(\cG \subset \La\) denote the set of all constant speed immersed closed geodesics in \(N\) with length at most \(L\). 
\begin{pro}\label{p.len.Ps.ga}
	For all \(\ve>0\), \tes \(\de>0\) \st if \(\ga\in \La\) satisfies \(\mathbf{d}\left(\ga,\cG\right)\geq \ve\) then \(\ell(\Ps(\ga))\leq \ell(\ga)-\de.\)
\end{pro}

\section{Relative homotopy groups of the free loop space}
Let \(X\) be a topological space, \(p\in X\). Let \(\bj:\Om_pX\ra \La X\) be the inclusion map, \(\bfe:\La X \ra X\) be defined by \(\bfe(\ga)=\ga(0)\) and \(\bc:X\ra \La X\) be defined by \(\bc(p)=c_p\). (For the notations, see Section \ref{s.notn}.) Then \(\bfe:\La X \ra X\) is a fibration; \(\Om_pX\) is the fiber over \(p\). Since \(\bfe\circ\bc=\text{id}_X\), the long exact sequence of homotopy groups, corresponding to the fibration \(\bfe\), splits, i.e. \fa \(m\geq 1\),
\begin{center}
	\begin{tikzcd}
	0 \arrow[r] & \pi_m(\Om_pX,c_p) \arrow[r, "\bj_*"] & \pi_m(\La X, c_p)\arrow[r, "\bfe_*"] & \pi_m(X,p) \arrow[r] \arrow[l, "\bc_*", dashed, shift left=2] & 0
	\end{tikzcd}
\end{center}
is a split exact sequence (see \cite{Moore}*{Section 2.3}). Hence, \(\pi_m(\La X, c_p)\) is isomorphic to the semi-direct product  \(\pi_m(\Om_p X, c_p)\rtimes \pi_m(X, p)\). In this section, we will prove a relative version of this theorem.

Let \(\cC\) denote the category of pointed sets. Then the objects of \(\cC\) are pairs \((S,s)\) where \(S\) is a set and \(s\in S\). (We call $s$ the distinguished element of $S$.) The morphisms of \(\cC\) are maps \(\vp:(S,s)\ra (S',s')\); we define the kernel of $\vp$ \(\ker(\vp)=\vp^{-1}(s')\). If $G$ is a group and \(1_G\in G\) is the identity element, then \((G,1_G)\) in an object of \(\cC\). Moreover, a group homomorphism maps the identity element to the identity element. Thus, the category of groups can be regarded as a subcategory of \(\cC\). In $\cC$, we have the notion of an exact sequence; 

\begin{center}
	\begin{tikzcd}
	\cdots \arrow[r] & (S_{i-1},s_{i-1}) \arrow[r, "\vp_{i-1}"] & (S_{i},s_{i}) \arrow[r, "\vp_{i}"] & (S_{i+1},s_{i+1}) \arrow[r, "\vp_{i-1}"] & \cdots
	\end{tikzcd}
\end{center}
is called exact if \(\text{im}(\vp_{i-1})=\ker(\vp_i)\) \fa $i$. 

For a topological space \(X\), \(\pi_0(X)\) denotes the set of path-components of \(X\). For \(p\in X\), \(\pi_0(X,p)\) denotes the pointed set $(\pi_0(X),X_p)$, where \(X_p\) is the path-component of \(X\) containing \(p\). Let \(\bj\) and \(\bfe\) be as defined above. Then 
\begin{center}
	\begin{tikzcd}
	\{0\} \arrow[r] & \pi_0(\Om_pX,c_p) \arrow[r, "\bj_*"] & \pi_0(\La X, c_p)\arrow[r, "\bfe_*"] & \pi_0(X,p) \arrow[r]  & \{0\}
	\end{tikzcd}
\end{center}
is an exact sequence. Another exact \seq of pointed sets, which will be useful for our purpose, is the homotopy exact sequence of a pair of pointed topological spaces. Let \(X\) be a topological space, \(A\subset X\) and \(p\in A\). Then the homotopy exact sequence of \((X,A,p)\) is the following long exact sequence \cite{Hat}*{Theorem 4.3}
\begin{multline}
\cdots \ra \pi_i(A,p) \ra \pi_i(X,p) \ra \pi_i(X,A,p) \ra \pi_{i-1}(A,p) \ra \cdots\\ \ra \pi_1(X,A,p) \ra \pi_{0}(A,p) \ra \pi_0(X,p).	
\end{multline}
All but the last three objects in this sequence are groups. The distinguished element of \(\pi_1(X,A,p)\) is the homotopy class of the constant map.

\begin{lem}\label{l.5.lem}
Let us consider the following commutative diagram of groups. 

\begin{center}
\begin{tikzcd}
1 \arrow[r] & C_1' \arrow[d, "f_1'"'] \arrow[r, "g_1"] & C_1 \arrow[r, "h_1"] \arrow[d, "f_1"'] & C_1'' \arrow[d, "f_1''"] \arrow[l, "j_1", dashed, shift left=2] \arrow[r] & 1 \\
1 \arrow[r] & C_2' \arrow[d, "f_2'"'] \arrow[r, "g_2"] & C_2 \arrow[r, "h_2"] \arrow[d, "f_2"'] & C_2'' \arrow[d, "f_2''"] \arrow[l, "j_2", dashed, shift left=2] \arrow[r] & 1 \\
& C_3' \arrow[d, "f_3'"'] \arrow[r, "g_3"] & C_3 \arrow[r, "h_3"] \arrow[d, "f_3"'] & C_3'' \arrow[d, "f_3''"] \arrow[l, "j_3", dashed, shift left=2]  &  \\
1 \arrow[r] & C_4' \arrow[d, "f_4'"'] \arrow[r, "g_4"] & C_4 \arrow[r, "h_4"] \arrow[d, "f_4"'] & C_4'' \arrow[d, "f_4''"]  \arrow[r] & 1 \\
1 \arrow[r] & C_5' \arrow[r, "g_5"]                    & C_5 \arrow[r, "h_5"]                   & C_5''                   \arrow[r] &  1
\end{tikzcd}   
\end{center}

We assume that
\begin{enumerate}[(i)]
	\item all the squares commute, i.e. for \(1\leq i\leq 4\), \(f_i\circ g_i = g_{i+1}\circ f'_i,\; f''_i\circ h_i= h_{i+1}\circ f_i\) and for $i=1,2$, \(f_i\circ j_i=j_{i+1}\circ f''_i;\)
	\item each column is exact;
	\item for \(i\neq 3\), the \(i\)-th row is a short exact sequence; 
	\item for $i=1,2$, short exact sequence in the $i$-th row splits, i.e. \(h_i\circ j_i = \textup{id}_{C_i''}\);
	\item for the third row, we have \(\textup{im}(g_3)\subset \ker(h_3)\) and \(h_3\circ j_3 = \textup{id}_{C_3''}\).
\end{enumerate}

Then the third row is a split exact sequence.

\end{lem}

Before we prove this lemma, let us recall the following fact about the split exact sequences.

\begin{lem}\label{l.split}
Let
\begin{center}
	\begin{tikzcd}
	1 \arrow[r] & C' \arrow[r, "g"] & C \arrow[r, "h"] & C'' \arrow[r] \arrow[l, "j", dashed, shift left=2] & 1
	\end{tikzcd}
\end{center}
be a split exact sequence of groups. Then \fa \(c\in C\), \te \(c'\in C'\) and \(c''\in C''\) \st \(c=g(c')j(c'')\).
\end{lem}

\begin{proof}[Proof of Lemma \ref{l.5.lem}]
Since \(h_3\circ j_3 = \textup{id}_{C_3''}\), \(h_3\) is surjective. Hence we only need to show that \(g_3\) is injective and \(\ker(h_3)\subset \text{im}(g_3)\). 

Suppose $c_3'\in \ker(g_3)$, i.e.
\begin{equation}\label{e.c'_3}
g_3(c'_3)=1.
\end{equation}
Then
\begin{align}
& g_4(f_3'(c'_3))= f_3(g_3(c_3')) = 1\\
& \implies f'_3(c'_3)=1\\
&\implies \exists\; c'_2 \in C'_2 \text{ \st } c'_3 = f'_2(c'_2) \label{e.5.3}\\
& \implies f_2(g_2(c'_2))=g_3(f'_2(c'_2)) = 1\;(\text{by }\eqref{e.c'_3})\\
& \implies \exists\; c_1 \in C_1 \text{ \st } g_2(c'_2)=f_1(c_1). \label{e.5.1}
\end{align}
Lemma \ref{l.split} implies that \te \(c'_1\in C'_1\) and \(c''_1\in C_1''\) \st \(c_1=g_1(c'_1)j_1(c_1'')\). Hence \eqref{e.5.1} implies
\begin{align}
& g_2(c_2')= f_1(g_1(c'_1))f_1(j_1(c_1'')) = g_2(f_1'(c_1'))j_2(f_1''(c_1''))\\
& \implies g_2(c_2'f_1'(c_1')^{-1}) = j_2(f_1''(c_1'')) \label{e.5.2}\\
& \implies h_2(g_2(c_2'f_1'(c_1')^{-1})) = h_2(j_2(f_1''(c_1''))) \\
& \implies 1= f_1''(c_1'') \\
& \implies g_2(c_2'f_1'(c_1')^{-1}) = 1\; (\text{by } \eqref{e.5.2}) \\
& \implies c_2' = f_1'(c_1')\\
& \implies c_3' = 1\; (\text{by } \eqref{e.5.3}).
\end{align}
\tf \(g_3\) is injective.

\sps \(d_3\in\ker(h_3)\), i.e.
\begin{equation}\label{e.f.0}
h_3(d_3)=1
\end{equation}
Then
\begin{align}
& h_4(f_3(d_3)) = f_3''(h_3(d_3))=1\\
& \implies \exists\; d_4'\in C_4' \text{ \st }f_3(d_3)=g_4(d_4') \label{e.f.1}\\
& \implies g_5(f_4'(d_4'))=f_4(g_4(d_4'))=1\\
& \implies f_4'(d_4') = 1\\
& \implies \exists\; d_3' \in C_3'\text{ \st } d_4'=f_3'(d_3')\\
& \implies f_3(d_3)= g_4(f_3'(d_3'))= f_3(g_3(d_3'))\;(\text{by } \eqref{e.f.1})\\
& \implies f_3(d_3g_3(d_3')^{-1})=1\\
& \implies \exists\; d_2\in C_2\text{ \st }d_3g_3(d_3')^{-1}=f_2(d_2) \label{e.f.2}.
\end{align}
Lemma \ref{l.split} implies that \te \(d_2'\in C_2'\) and \(d_2''\in C_2''\) \st \(d_2=g_2(d_2')j_2(d_2'')\). \hn \eqref{e.f.2} implies
\begin{align}
& d_3g_3(d_3')^{-1}=f_2(d_2) = f_2(g_2(d_2'))f_2(j_2(d_2''))\\
& \implies d_3g_3(d_3')^{-1}=g_3(f_2'(d_2'))j_3(f_2''(d_2'')) \label{e.f.3}.
\end{align}
Applying \(h_3\) on the both sides of \eqref{e.f.3} and using \eqref{e.f.0}, we obtain
\[1=f_2''(d_2'');\]
hence \eqref{e.f.3} implies 
\[d_3=g_3(f_2'(d_2')d_3').\]
\tf \(\ker(h_3)\subset \text{im}(g_3)\).
\end{proof}

\begin{rmk}\label{r.5.lem}
In Lemma \ref{l.5.lem}, instead of assuming that all the objects in the diagram are groups, let us assume that the objects in the 1st and 2nd row are groups; the objects in the 3rd, 4th and 5th row are pointed sets and the exactness conditions hold in the sense of pointed sets. Then, from the proof of Lemma \ref{l.5.lem}, it follows that \(\ker(g_3)\) is a singleton set.
\end{rmk}

\begin{pro}\label{p.rel.htp}
	Let \(X\) be a topological space, \(A\subset X\) and $p\in A$. Then, for all \(m\geq 2\), the relative homotopy group \(\pi_m(\La X, \La A, c_p)\) is isomorphic to the semi-direct product \(\pi_m(\Om_p X,\\ \Om_p A, c_p)\rtimes \pi_m(X,A, p)\). Moreover, for $m=1$, the map \(\pi_1(\Om_p X, \Om_p A, c_p) \ra \pi_1(\La X, \La A, c_p)\), which is induced by the inclusion \((\Om_p X, \Om_p A)\hookrightarrow (\La X, \La A)\), is non-constant. 
\end{pro}
\begin{proof} For \(m\geq 1\), let us consider the following commutative diagram.
\begin{center}
	\begin{tikzcd}
	0 \arrow[r] & \pi_m(\Om_p A) \arrow[d] \arrow[r] & \pi_m(\La A) \arrow[r] \arrow[d] & \pi_m(A) \arrow[d] \arrow[l,  dashed, shift left=2] \arrow[r] & 0 \\
	0 \arrow[r] & \pi_m(\Om_p X) \arrow[d] \arrow[r] & \pi_m(\La X) \arrow[r] \arrow[d] & \pi_m(X) \arrow[d] \arrow[l,  dashed, shift left=2] \arrow[r] & 0 \\
	& \pi_m(\Om_p X, \Om_p A) \arrow[d] \arrow[r] & \pi_m(\La X, \La A) \arrow[r] \arrow[d] & \pi_m(X,A) \arrow[d] \arrow[l, dashed, shift left=2]  &  \\
    0 \arrow[r] & \pi_{m-1}(\Om_p A) \arrow[d] \arrow[r] & \pi_{m-1}(\La A) \arrow[r] \arrow[d] & \pi_{m-1}(A) \arrow[d]  \arrow[r] & 0 \\
	0 \arrow[r] & \pi_{m-1}(\Om_p X) \arrow[r] & \pi_{m-1}(\La X) \arrow[r] & \pi_{m-1}(X)  \arrow[r] & 0 \\
	\end{tikzcd}   
\end{center}

The columns correspond to the homotopy exact sequence for the pairs \((\Om_pX, \Om_p A)\), \((\La X, \La A)\) and \((X,A)\) with the basepoints $c_p$, $c_p$ and $p$. The rows (except the 3rd row) correspond to the short exact sequence for the loop space fibration (as discussed at the beginning of this section). The maps in the 3rd row too
\begin{center}
	\begin{tikzcd}
	 \pi_m(\Om_p X,\Om_p A) \arrow[r, "\bj_*"] & \pi_m(\La X, \La A)\arrow[r, "\bfe_*"] & \pi_m(X,A) \arrow[l, "\bc_*", dashed, shift left=2]
	\end{tikzcd}
\end{center}
are induced by the maps \(\bj\), \(\bfe\) and \(\bc\); from the definitions of these maps, it follows that \(\text{im}(\bj_*)\subset \ker(\bfe_*)\) and \(\bfe_*\circ\bc_*=\text{id}_{\pi_m(X,A)}\). 

If \(m\geq 2\), all the objects in this diagram are groups. \tf by Lemma \ref{l.5.lem}, the 3rd row is a split exact sequence; hence \(\pi_m(\La X, \La A, c_p)\cong \pi_m(\Om_p X, \Om_p A, c_p)\rtimes \pi_m(X,A, p)\). If \(m=1\), the objects in the 1st and 2nd row are groups; however the objects in the 3rd, 4th and 5th row are only pointed sets. \tf by Remark \ref{r.5.lem}, the kernel of the map \(\bj_*\) in the 3rd row is a singleton set. Since \(\bj_*\) maps the homotopy class of the constant map in \(\pi_1(\Om_p X, \Om_p A, c_p)\) to the homotopy class of the constant map in \(\pi_1(\La X, \La A, c_p)\), \(\bj_*:\pi_1(\Om_p X, \Om_p A, c_p) \ra \pi_1(\La X, \La A, c_p)\) must be a non-constant map.

\end{proof}
\begin{rmk}
In the above commutative diagram, which appears in the proof of Proposition \ref{p.rel.htp}, the 4th and 5th row are split exact \fa $m\geq 1$ and all the objects are Abelian groups when $m\geq 3$. In Lemma \ref{l.5.lem}, if we additionally assume that the 4th and 5th row are split exact and all the objects are Abelian groups, then the lemma follows from the five lemma for the Abelian groups.
\end{rmk}
\begin{rmk}\label{r.rel.htp}
For all \(m\geq 1\), \tes a canonical bijection \(D:\pi_m(\Om_pX,\Om_pA,c_p)\ra\pi_{m+1}(X,A,p)\), which is an isomorphism if \(m\geq 2\). Let us briefly recall the definition of \(D\). 

Let $I$ denote the interval $[0,1]$. Let $\sI^0=\{1\}$ and for \(j\geq 2\),
\[\sI^{j-1}=\ov{\del I^j\setminus (I^{j-1}\times \{0\})}.\]
Let
\[\tilde{\sI}^1=\ov{\del I^2\setminus (\{0\}\times I)} \text { and for } j\geq 3,\; \tilde{\sI}^{j-1}=\ov{\del I^j \setminus (I^{j-2}\times\{0\}\times I)}.\]
Given \(\Xi:(I^m,\del I^m,\sI^{m-1})\ra (\Om_pX,\Om_pA,\{c_p\})\), we define \(\xi:(I^{m+1},\del I^{m+1},\tilde{\sI}^{m})\ra(X,A,\{p\})\) by \(\xi(s_1,\dots,s_m,s_{m+1})=\Xi(s_1,\dots,s_m)(s_{m+1})\). Then \(D:\pi_m(\Om_pX,\Om_pA,c_p)\ra\pi_{m+1}(X,A,p)\), defined by \(D([\Xi])=[\xi]\), is a well-defined map, which is a group homomorphism if \(m\geq 2\). The inverse of the map $D$, which we denote by \(\De\), is defined as follows. Given \(\xi':(I^{m+1},\del I^{m+1},\tilde{\sI}^{m})\ra(X,A,\{p\})\), we define \(\Xi':(I^m,\del I^m,\sI^{m-1})\ra (\Om_pX,\Om_pA,\{c_p\})\) by \(\Xi'(s_1,\dots s_m)=\xi'(s_1,\dots,s_m,-)\). Then \(\De: \pi_{m+1}(X,A,p) \ra \pi_m(\Om_pX,\Om_pA,c_p)\) is defined by $\De([\xi'])=[\Xi']$.
\end{rmk}

\section{Proof of Theorems \ref{thm} and \ref{thm.main}}\label{s.proof}
To prove Theorem \ref{thm}, we need to prove the existence of a closed geodesic on $M$ when one of the following conditions are satisfied:
\begin{itemize}
	\item[(C1)] \(\pi_1(M,M\setminus U)=0\);
	\item[(C2)] \(\pi_0(\La(M\setminus U))\ra \pi_0(\La M)\) is not injective;
	\item[(C3)] \(\pi_m(\La (M\setminus U), c^*)\ra \pi_m(\La M, c^*)\) is not an isomorphism for some \(m\geq 1\) and \(c^* \in \La (M\setminus U)\) which is non-contractible in $M\setminus U$;
	\item[(C4)] \(\pi_0(\La(M\setminus U))\ra \pi_0(\La M)\) is not surjective.
\end{itemize}
\begin{lem}\label{l.pi.k}
Let \(X\) be a connected topological manifold. Suppose \(A\) is a connected open subset of \(X\) \sot \(\overline{A}\) is a compact manifold \w boundary \(\del A\). Then the relative homotopy group \(\pi_i(X,X\setminus A)\neq 0\) for some \(i\geq 1.\)
\end{lem}
\begin{proof}
\sps \(\pi_i(X,X\setminus A)= 0\) \fa \(i\geq 1.\) Then by the Hurewicz theorem \cite{Hat}*{Theorem 4.37}, the relative homology groups
\begin{equation}\label{e.H.1}
H_i(X, X \setminus A)=0\quad \forall\; i\geq 1.
\end{equation}
By the excision theorem \cite{Hat}*{Theorem 2.20}, $H_i(X, X \setminus A)\cong H_i(\cN_r(A),\cN_r(A)\setminus A)$ (for all $r>0$). By \cite{Hat}*{Proposition 3.42}, if $r$ is sufficiently small, \((\cN_r(A),\cN_r(A)\setminus A)\) deformation retracts onto \((\ov{A},\del A)\); hence \(H_i(\cN_r(A),\cN_r(A)\setminus A) \cong H_i(\ov{A},\del A)\). \tf \eqref{e.H.1} implies
\begin{equation}
H_i(\ov{A},\del A)=0\quad \forall\; i\geq 1.
\end{equation}
However, by the Lefschetz duality \cite{Hat}*{Theorem 3.43}, \(H_n(\ov{A},\del A,\bbz_2)\neq 0\) (where $n=\dim(A)$); hence by the universal coefficient theorem for homology \cite{Hat}*{Theorem 3A.3}, either \(H_{n-1}(\ov{A},\del A)\neq 0\) or \(H_n(\ov{A},\del A)\neq 0\).
\end{proof}

Let $\sM$ be a complete \Rm \mf and \(\ga\in \sL\sM\). We define the Radon measure \(\cL_{\ga}\) on \([0,1]\) as follows. If \(A\) is a Borel subset of \([0,1]\),
\[\cL_{\ga}(A)=\int_A\md{\ga'(s)}\;ds.\]
The Radon measure \(\mu_{\ga}\) on \(\sM\) is defined by \(\mu_{\ga}=\ga_{*}\cL_{\ga}\), i.e. if \(B\) is a Borel subset of \(\sM\), then \(\mu_{\ga}(B)=\cL_{\ga}(\ga^{-1}(B))\).

\begin{lem}\label{l.mu.ga}
The map \(\ga\mapsto \mu_{\ga}\) is a \cts map from \(\sL\sM\) to the space of Radon measures on \(\sM\).
\end{lem}
\begin{proof}\newcommand{\gi}{\gamma_i}\newcommand{\gin}{\gamma_{\infty}}
Let \(\{\gi\}_{i=1}^{\infty}\) be a \seq in \(\sL\sM\) which converge to \(\gin\) and \(\vp\in C^0_c(\sM)\). Then
\[\int_{\sM} \vp \; d\mu_{\ga}=\int_{0}^{1}\vp(\ga(s))\md{\ga'(s)}\;ds.\]
Thus,
\begin{align}
&\md{\int_{\sM} \vp \; d\mu_{\gin}-\int_{\sM} \vp \; d\mu_{\gi}}\\
&\leq \int_{0}^{1}\Big|\vp(\gin(s))\md{\gin'(s)}-\vp(\gi(s))\md{\gi'(s)}\Big|\;ds\\
&\leq \int_{0}^{1}\md{\vp(\gin(s))-\vp(\gi(s))}\md{\gin'(s)}\;ds + \int_{0}^{1}\md{\vp(\gi(s))}\Big|\md{\gin'(s)}-\md{\gi'(s)}\Big|\;ds.\label{e.gi}
\end{align}
Let us assume that \(\sM\) is isometrically embedded in the Euclidean space $\bbr^P$. For \(x\in \sM\), the tangent space \(T_x\sM\) is a subspace of \(T_x\bbr^P\) and \(T_x\bbr^P\) can be canonically identified with \(\bbr^P\). Using this identification, let us set
\[\de_i=\left(\int_{0}^{1}\md{\gin'(s)-\gi'(s)}^2\;ds\right)^{1/2}.\]
Let
\[\ve_i=\sup_{s\in [0,1]} \md{\vp(\gin(s))-\vp(\gi(s))}.\]
Since \(\ga_i\ra \gin\) in \(\sL\sM\), there holds \(\ve_i,\de_i\ra 0\). Hence, from \eqref{e.gi}, we obtain
\[\md{\int_{\sM} \vp \; d\mu_{\gin}-\int_{\sM} \vp \; d\mu_{\gi}}\leq \ve_i\ell(\gin)+\de_i\nm{\vp}_{C^0}\ra 0.\]
\end{proof}

\begin{proof}[Proof of Theorem \ref{thm} when (C1) holds]
The proof is divided into eight parts.

\noindent\textbf{Part 1.} Since we are assuming that $\pi_1(M,M\setminus U)=0$, by Lemma \ref{l.pi.k}, \tes \(k\geq 1\) \st 
\begin{equation}\label{e.pi.M.U}
\pi_i(M,M\setminus U)=0 \; \forall \; 1\leq i \leq k
\end{equation}
 and \(\pi_{k+1}(M,M\setminus U)\neq 0\). \tf by Proposition \ref{p.rel.htp} and Remark \ref{r.rel.htp}, \(\pi_{k}(\La M,\La (M\setminus U),c_p)\neq 0\). Let us choose \(\tilde{F}:(I^k,\del I^k)\ra (\La M, \La (M\setminus U))\) \sot 
 \begin{equation}\label{e.[F].nonzero}
 [\tilde{F}]\neq 0\in \pi_k(\La M, \La (M\setminus U),c_p).
 \end{equation}
 Let \(\tilde{f}:I^k\times S^1 \ra M\) be defined by \(\tilde{f}(v,s)=\tilde{F}(v)(s)\). 
 
 Let \(\rh>0\) \st \(M\setminus U\) is \(\rh\)-convex (Definition \ref{d.conv}). Let us choose \(K\in \bbn\) \sot \(\md{v-v'}+\md{s-s'}\leq 1/K\) implies 
 \begin{equation}\label{e.K.rh}
 d\left(\tilde{f}(v,s), \tilde{f}(v',s')\right)\leq \frac{\rh}{2}
 \end{equation} 
 and
 \[d\left(\tilde{f}(v,s), \tilde{f}(v',s')\right)<\frac{1}{2}\inf\{\text{conv}_p(M):p\in \text{image}(\tilde{f})\}.\]
 We define \(F:(I^k,\del I^k)\ra (\sL M, \sL (M\setminus U))\) as follows. For $v\in I^k$, $F(v):S^1\ra M$ is defined by the condition that \fa \(i\in \{0,1,2,\dots K\}\), \(F(v)(i/K)=\tilde{F}(v)(i/K)\) and \fa $i\in [K]$, \(F(v)\) restricted to \([\frac{i-1}{K},\frac{i}{K}]\) is a minimizing constant speed geodesic. Then $F(v)\in \sL M$ and $F:I^k\ra \sL M$ is continuous. Moreover, by \eqref{e.K.rh}, $F(\del I^k)\subset \sL (M\setminus U)$. Let us define the homotopy \(G_0:(I^k,\del I^k)\times I\ra (\La M, \La (M\setminus U))\) by 
 \[G_0(v,t)(s)=\bH(\tilde{F}(v)(s),F(v)(s))(t),\]
 where \(\bH\) is as in \eqref{e.bH}. Then \(G_0(-,0)=\tilde{F}\) and \(G_0(-,1)=F\). Moreover, by \eqref{e.K.rh}, if $v\in \del I^k$, then $G_0(v,t)\in \La (M\setminus U)$ \fa $t\in I$.

\noindent\textbf{Part 2.} Let \(f:I^k\times S^1 \ra M\) be defined by \(f(v,s)=F(v)(s)\). Let
\begin{equation}\label{e.la.E}
\la=\sup\{\ell(F(v)):v\in I^k\},\quad E=\sup\{\sE(F(v)):v\in I^k\},
\end{equation}
$V=\cN_{2\la}(U)$,
\begin{equation}\label{e.al}
\al=\inf \{\text{conv}_p(M):p\in V \}
\end{equation}
and $W\subset M$ be a bounded open set which contains \(\cN_{\al}(V)\:\cup\:\text{image}(f)\). Let us choose a closed \Rm \mf \(N\) \st \(N\) and \(M\) have the same dimension and \tes an isometric embedding \(\iota:W \ra N\). We will abuse notation by identifying \(W\) \w \(\iota(W)\).

\noindent\textbf{Part 3.} Let $\rh$ be as in Part 1 and $E$ be as in \eqref{e.la.E}. We choose \(R>0\) and \(L\in \bbn\) \sot 
\begin{equation}\label{e.R.rh}
R<\min\left\{\frac{\al}{4},\frac{\rh}{2} \right\}
\end{equation}
and \((E,R,L)\) satisfies \eqref{e.R}. For this fixed choice of \((E,R,L)\), let us consider the maps \(\Ps\) and \(\Ph\) as defined in Definition \ref{d.Ps} and Proposition \ref{p.Ph}.

\begin{lem}\label{l.F.G}
\Te sequences \(\{F_i:(I^k,\del I^k)\ra (\sL N, \sL (N\setminus U))\}_{i=0}^{\infty}\) and \(\{G_i:(I^k,\del I^k)\times I\ra (\sL N, \sL (N\setminus U))\}_{i=1}^{\infty}\) \w the following properties. $F_0=F$ and \fa $i\geq 1$
\begin{enumerate}[(i)]
	\item \(G_i(-,0)=F_{i-1},\; G_i(-,1)=F_i\);
	\item $\md{G_i(v,t)}\subset W\;\forall\;(v,t)\in I^k\times I$;
	\item for $v\in I^k$, if \(\md{F_{i-1}(v)}\cap U\neq \emptyset\), then \(F_i(v)=\Ps\left(F_{i-1}(v)\right)\);
	\item for $v\in I^k$, if \(\md{F_{i-1}(v)}\cap U= \emptyset\), then \(\md{F_{i}(v)}\cap U= \emptyset\) as well;
	\item for $v\in I^k$, if \(\md{F_i(v)}\cap U\neq \emptyset\), then \(\ell(F_i(v))\leq \la\).
\end{enumerate}
\end{lem}
\begin{proof}
We set \(F_0=F\). By the choice of \(W\), \(\md{F(v)}\subset W\) \fa \(v\in I^k\). Moreover, by the definition of \(\la\), \(\ell(F(v))\leq \la\) \fa \(v\in I^k\). Let us assume that \(F_j:(I^k,\del I^k)\ra (\sL N, \sL (N\setminus U))\) has already been defined for some \(j\geq 0\) and $F_j$ has the following properties:
\begin{itemize}
	\item \(\md{F_j(v)}\subset W\;\forall v\in I^k\);
	\item for $v\in I^k$, if \(\md{F_j(v)}\cap U\neq \emptyset\), then \(\ell(F_j(v))\leq \la\).
\end{itemize}
We define \(F_{j+1}:(I^k,\del I^k)\ra (\sL N, \sL (N\setminus U))\) and \(G_j:(I^k,\del I^k)\times I\ra (\sL N, \sL (N\setminus U))\) as follows. Let
\[D^0_j=\{v\in I^k: \md{F_j(v)}\cap (N\setminus V)\neq \emptyset\};\quad D^1_j=\ov{\{v\in I^k: \md{F_j(v)}\cap U\neq \emptyset\}}.\]
Then, \(D^0_j,\;D^1_j\) are closed subsets of \(I^k\). Moreover, by our assumption, if \(v\in D^1_j\), then
\[\md{F_j(v)}\subset \ov{\cN_{\la}(U)};\]
\hn \(D^0_j\:\cap\: D^1_j=\emptyset\). Let \(\ch_j:I^k\ra [0,1]\) be a \cts \fn \st 
\begin{equation}
\ch_j=\begin{cases}
0 &\text{ on } D^0_j;\\
1 &\text{ on } D^1_j.
\end{cases}
\end{equation}
For $(v,t)\in I^k\times I$, we define
\[G_{j+1}(v,t)=\Ph(F_j(v),\ch_j(v)t),\quad F_{j+1}(v)=\Ph(F_j(v),\ch_j(v)).\]
Then $G_{j+1}(-,0)=F_j$, $G_{j+1}(-,1)=F_{j+1}$. Remark \ref{r.spt.Ps.Ph} and \eqref{e.R.rh} imply that \(G_{j+1}(v,t)\in \sL (N\setminus U)\) \fa \((v,t)\in \del I^k\times I\) and for $v\in I^k$, if \(\md{F_{j}(v)}\cap U= \emptyset\), then \(\md{F_{j+1}(v)}\cap U= \emptyset\). \tf if \(\md{F_{j+1}(v)}\cap U\neq \emptyset\) then \(\md{F_{j}(v)}\cap U \neq \emptyset\). From the definition of $F_{j+1}$, it follows that if \(\md{F_{j}(v)}\cap U \neq \emptyset\), then \(F_{j+1}(v)=\Ps(F_j(v))\); in that case by Proposition \ref{p.len.Ps.ga}, $\ell(F_{j+1}(v))\leq \ell(F_{j}(v))\leq \la$. 

It remains to show that \(\md{G_{j+1}(v,t)}\subset W\) \fa \((v,t)\in I^k\times I\). If \(v\in D^0_j\), then \(G_{j+1}(v,t)=F_j(v)\) \fa \(t\in I\); hence \(\md{G_{j+1}(v,t)}=\md{F_j(v)}\subset W\). If \(v\notin D^0_j\), then \(\md{F_j(v)}\subset V\). \tf by Remark \ref{r.spt.Ps.Ph} and \eqref{e.R.rh},
\[\md{G_{j+1}(v,t)}\subset \cN_{4R}\left(\md{F_j(v)}\right)\subset \cN_{\al}(V)\subset W. \]
\end{proof}

\noindent\textbf{Part 4.} \Sn $U$ is an open set \w $C^0$ boundary $\del U$, \tes an open set \(V'\subset \cN_{\la}(U)\) and a homeomorphism \(\cF:\del U\times (-1,1)\ra V'\) \st \(\cF(x,0)=x\) \fa \(x\in \del U\);
\begin{equation}
\cF(\del U \times (-1,0)) \subset U;\quad  \cF(\del U \times (0,1)) \subset V'\setminus \ov{U}.
\end{equation}

Let \(d_{\del U}\) denote the signed distance from \(\del U\) \st \(d_{\del U}<0\) on \(U\). For \(r\in \bbr\), let 
\[U(r)=\{p\in N: d_{\del U}(p)<r\}.\] We choose \(0<\et\leq \al \) (where $\al$ is as in \eqref{e.al}) \sot 
\begin{equation}\label{e.et}
\{-3\et \leq d_{\del U}\leq \et\} \subset V'.
\end{equation}

For \(\ga\in\sL N\), let \(\mu_{\ga} \) be the Radon measure as defined before Lemma \ref{l.mu.ga}; if $B$ is a Borel subset of \(N\), we use the notation
\[\ell(\ga,\;B)=\mu_{\ga}(B). \]

\begin{clm}\label{c.geodesic}
	\Tes a closed geodesic \(\si\) in \(N\) \st \(\ell(\si)\leq \la\) and \(\ell\left(\si,\;\ov{U(\et)}\right)\geq \et\) .
\end{clm}
The proof of this claim will be presented in Parts 5 -- 8. Assuming the claim, we obtain
\[\md{\si}\subset \cN_{\la}(U(\et))\subset V.\]
Viewing \(V\) as a subset of \(M\) we conclude that \(\si\) is a non-trivial closed geodesic in \(M\).

\noindent\textbf{Part 5.} Let \(\sG \subset \sL N\) denote the set of all constant speed immersed closed geodesics in \(N\) with length at most \(\la\). We assume that \fa \(\be\in \sG\), there holds 
\begin{equation}\label{e.mu.beta}
\ell\left(\be,\;\ov{U(\et)}\right)<\et
\end{equation} 
\(\sG\subset \sL N\) is compact; by Lemma \ref{l.mu.ga},
\[\cS=\left\{c\in\sL N:\ell(c)\leq \la,\;\ell\left(c,\;\ov{U(\et)}\right)\geq \et\right\} \]
is a closed subset of $\sL N$. \hn \eqref{e.mu.beta} implies that \tes \(\ta>0\) \st \fa $\ga\in\cS$, \(\mathbf{d}(\ga,\sG)\geq \ta\) (where $\mathbf{d}$ is as in Proposition \ref{p.len.Ps.ga}). \tf by Proposition \ref{p.len.Ps.ga}, \tes \(\ze>0\) \st if \(\ga\in \La\;\cap\:\cS\), then 
\begin{equation}\label{e.ze}
\ell(\Ps(\ga))\leq \ell(\ga)-\ze.
\end{equation}
We choose $Q\in \bbn$ \st \(Q\geq \frac{\la}{\ze} +1.\)
\begin{lem}\label{l.F_Q(v)}
For all \(v\in I^k\), one of the following two alternatives hold.
\begin{enumerate}[(i)]
	\item \(\md{F_Q(v)}\cap U=\emptyset\).
	\item \(\md{F_Q(v)}\subset U(\et)\) and \(\ell(F_Q(v))<\et\).
\end{enumerate}
\end{lem}

\begin{proof}
\sps \tes \(v\in I^k\) \st \fa \(0\leq i\leq Q\),
\begin{equation}\label{e.F_i.0.Q}
\md{F_i(v)}\cap U \neq \emptyset\text{ and } \ell\left(F_i(v),\;\ov{U(\et)}\right)\geq \et.
\end{equation}
Then, by Lemma \ref{l.F.G}(iii), \(F_{i+1}(v)=\Ps(F_i(v))\) \fa \(0\leq i \leq Q-1\). \hn by \eqref{e.ze},
\begin{equation}
\ell(F_{i+1}(v))\leq \ell(F_i(v))-\ze\quad\forall\;1\leq i \leq Q-1.
\end{equation}
This implies
\[\ell(F_Q(v))\leq \ell(F_1(v))-(Q-1)\ze\leq \la-(Q-1)\ze\leq 0,\]
which contradicts \eqref{e.F_i.0.Q}. \tf \fa \(v\in I^k\), either 
\begin{equation}\label{e.F_j.U}
\md{F_j(v)}\cap U=\emptyset \text{ for some }0\leq j\leq Q
\end{equation}
or
\begin{equation}\label{e.l.F_j.U'}
\md{F_i(v)}\cap U \neq \emptyset\;\forall\;0\leq i\leq Q\text{ and } \ell\left(F_j(v),\ov{U(\et)}\right)<\et \text{ for some }0\leq j\leq Q.
\end{equation}
If \eqref{e.F_j.U} holds, then by Lemma \ref{l.F.G}(iv), \(\md{F_Q(v)}\cap U=\emptyset.\) If \eqref{e.l.F_j.U'} holds, then \(\md{F_j(v)}\subset U(\et)\), which implies \(\ell(F_j(v))<\et\). \tf by Lemma \ref{l.F.G}(iii), \(\ell(F_Q(v))<\et\) which implies \(\md{F_Q(v)}\subset U(\et)\).
\end{proof}

\noindent\textbf{Part 6.} Let \(\Tht:\left\{ c\in \sL N: \ell(c)<\text{conv}_{c(0)}(N) \right\}\times [0,1]\ra \sL N\) be defined by
\[\Tht(\ga, t)(s)=\bH(\ga(s), \ga(0))(t),\]
where \(\bH\) is as in \eqref{e.bH}. Then \(\Tht(\ga,0)=\ga,\;\Tht(\ga,1)=c_{\ga(0)}\) and 
\begin{equation}\label{e.spt.Tht}
\md{\Tht(\ga,t)}\subset B(\ga(0),\ell(\ga))\quad \forall\;t\in [0,1].
\end{equation}
Let
\[D^0=\left\{v\in I^k:\md{F_Q(v)}\subset N\setminus U(-\et)\right\},\;D^1=\left\{ v\in I^k : \md{F_Q(v)}\cap \ov{U(-2\et)}\neq \emptyset \right\}.\]
Then $D^0$, \(D^1\) are disjoint closed subsets of $I^k$. Let \(\ch:I^k\ra [0,1]\) be a \cts \fn \st
\begin{equation}
\ch(v)=\begin{cases}
0 &\text{ if } v\in D^0;\\
1 & \text{ if } v\in D^1.
\end{cases}
\end{equation}
By Lemma \ref{l.F_Q(v)}, if \(\md{F_Q(v)}\cap \ov{U(-\et)}\neq \emptyset\) then \(\ell(F_Q(v))<\et\); hence \(\md{F_Q(v)}\subset U\). We define \(G':(I^k,\del I^k)\times I\ra (\sL N, \sL (N\setminus U))\) by
\begin{equation}
G'(v,t)=\begin{cases}
\Tht\left(F_Q(v),\ch(v)t\right) &\text{ if } \md{F_Q(v)}\cap \ov{U(-\et)}\neq \emptyset;\\
F_Q(v) &\text{ if }\md{F_Q(v)}\subset N\setminus U(-\et).
\end{cases}
\end{equation}
Then \(G'(v,0)=F_Q(v)\) and we define $F': (I^k,\del I^k) \ra (\sL N, \sL (N\setminus U))$ by \(F'(v)=G'(v,1)\). Moreover, if \(G'(v,t)\neq F_Q(v)\), then \eqref{e.spt.Tht} implies \(\md{G'(v,t)}\subset U(\et)\); therefore, by Lemma \ref{l.F.G}(ii), 
\begin{equation}\label{e.spt.G'}
\md{G'(v,t)}\subset W\;\forall\;(v,t)\in I^k\times I.
\end{equation} 
We also note that, if \(F'(v)\) is non-constant, then \(\md{F_Q(v)}\subset N\setminus \ov{U(-2\et)}\), which implies
\begin{equation}\label{e.F'.non-cnst}
\md{F'(v)}\subset N\setminus U(-3\et).
\end{equation}

\noindent\textbf{Part 7.} Given a \cts map \(\ph:N \ra N \) and \(\ga \in \La N\), let \(\ph_{\#}\ga\in \La N\) be defined by \(\left(\ph_{\#}\ga\right)(s)=\ph(\ga(s)). \) \eqref{e.et} implies that \tes \(\tht \in (0,1)\) \st 
\begin{equation}\label{e.tht}
\{-3\et \leq d_{\del U}\leq \et\} \subset \cF(\del U\times [-1+\tht,1-\tht]).
\end{equation}
Let \(\ps:\bbr\ra [0,1]\) be a compactly supported \(C^1\) \fn \sot \(\text{spt}(\ps)\subset (-1,1)\) and \(\ps\equiv 1\) on \([-1+\tht,1-\tht]\). Let \(\{\hat{\ph}^t\}_{t\geq 0}\) denote the flow of the vector field $\ps(\xi)\del_{\xi}$ on $\bbr$. Then \tes \(T>0\) \st 
\begin{equation}\label{e.T}
\hat{\ph}^T([-1+\tht,\infty)) \subset [1-\tht,\infty).
\end{equation}
For \(t\geq 0\), we define \(\ph^t:N\ra N\) as follows. If \(p\in V'\) and \(p=\cF(x,\xi)\), we set \(\ph^t(p)=\cF(x,\hat{\ph}^t(\xi))\); if \(p\notin V'\), we set \(\ph^t(p)=p\). Then \(\{\ph^t\}_{t\geq 0}\) is a \cts family of homeomorphisms of \(N\). \eqref{e.tht} and \eqref{e.T} imply
\begin{equation}\label{e.F^T}
\ph^T\left( N\setminus U(-3\et) \right) \subset N\setminus U(\et).
\end{equation}
We define $F'': (I^k,\del I^k) \ra (\La N, \La (N\setminus U))$ and \(G'':(I^k,\del I^k)\times I\ra (\La N, \La (N\setminus U))\) by \(F''(v)=\ph ^{T}_{\#}F'(v)\) and \(G''(v,t)=\ph ^{Tt}_{\#}F'(v)\). Then \(G''(-,0)=F'\) and \(G''(-,1)=F''\). Moreover, since \(\ph^t(p)=p\) \fa \(t\geq 0\) and \(p\in N\setminus V'\), by \eqref{e.spt.G'}, 
\begin{equation}\label{e.spt.G''}
\md{G''(v,t)}\subset W\;\forall\;(v,t)\in I^k\times I.
\end{equation}  
Furthermore, by \eqref{e.F'.non-cnst} and \eqref{e.F^T}, if \(F''(v)\) is non-constant, then 
\begin{equation}\label{e.F''.non-cnst}
\md{F''(v)}\subset N\setminus U(\et).
\end{equation}

Concatenating the homotopies $\{G_i\}_{i=0}^Q$, $G'$ and $G''$, we conclude that \tes a homotopy \(G: (I^k,\del I^k)\times I\ra (\La N, \La (N\setminus U))\) \st \(G(-,0)=\tilde{F}\) and \(G(-,1)=F''\); Lemma \ref{l.F.G}(ii), \eqref{e.spt.G'} and \eqref{e.spt.G''} imply that 
\begin{equation}\label{e.spt.G}
\md{G(v,t)}\subset W\;\forall\;(v,t)\in I^k\times I.
\end{equation}  
Viewing \(W\) as a subset of \(M\), we can regard \(G\) as a map from $(I^k,\del I^k)\times I$ to $(\La M, \La (M\setminus U))$ and \(F''\) as a map from $(I^k,\del I^k)$ to $(\La M, \La (M\setminus U))$. \tf using \eqref{e.[F].nonzero}, we obtain
\begin{equation}\label{e.[F''].nonzero}
[F'']\neq 0 \in \pi_k(\La M,\La (M\setminus U),c_p).
\end{equation}

\noindent\textbf{Part 8.} For \(j\in \bbn\), let \(I(j)\) be the cell complex on \(I\) whose \(0\)-cells are
\[[0],[j^{-1}],\dots,[1-j^{-1}],[1]\]
and \(1\)-cells are
\[[0,j^{-1}],[j^{-1},2j^{-1}],\dots, [1-j^{-1},1].\]
Let \(I^k(j)\) denote the cell complex on \(I^k\) whose cells are \(c_1\otimes c_2\otimes\dots\otimes c_k\), where each \(c_i\in I(j)\). By abuse of notation, a cell \(c_1\otimes c_2\otimes\dots\otimes c_k\) will be identified with its support \(c_1\times c_2\times\dots\times c_k \subset I^k\).

We recall that \(\mathbf{e}:\La M \ra M\) is defined by \(\mathbf{e}(\ga)=\ga(0)\). Let \(f'':(I^k,\del I^k) \ra (M, M\setminus U)\) be defined by \(f''(v)=\mathbf{e}(F''(v)).\) Let
\[\left(f''\right)^{-1}(U)=\cU,\quad \left(f''\right)^{-1}(U(\et))=\tilde{\cU}. \]
We choose \(l\in \bbn\) \sot if \(c\in I^k(l)\) and \(c\cap \cU\neq \emptyset\) then \(c \subset \tilde{\cU}\). Let \(\bX \subset I^k(l)\) be the union of all top dimensional cells \(c\in I^k(l)\) (i.e. \(\dim(c)=k\)) \st \(c\cap \cU\neq \emptyset\). Let \(\bA\) be the union of all cells \(e\in \bX\) \st \(e\cap \cU = \emptyset\). Then
\begin{equation}\label{e.A}
\ov{(I^k(l)\setminus \bX)}\cap \bX \subset \bA.
\end{equation}
Since $M\setminus U$ is connected, using \eqref{e.pi.M.U} and applying \cite{Hat}*{Lemma 4.6} to the map \(f'':(\bX,\bA)\ra (M,M\setminus U)\), we conclude that \te \(\hat{h}:(\bX,\bA)\times I \ra (M,M\setminus U)\) and \(\hat{f}: \bX \ra M\setminus U\) \st \(\hat{h}(-,0)=f''\), \(\hat{h}(-,1)=\hat{f}\) and \(\hat{h}(v,t)=f''(v)\) \fa \((v,t)\in \bA\times I\). We define \(\hat{H}:(\bX,\bA)\times I \ra (\La M,\La (M\setminus U))\) and \(\hat{F}: \bX \ra \La (M\setminus U)\) by 
\[\hat{H}(v,t)=c_{\hat{h}(v,t)}\; \text{ and }\; \hat{F}(v)=c_{\hat{f}(v)}.\] 
\eqref{e.F''.non-cnst} implies \(F''(v)=c_{f''(v)}\) \fa \(v\in \tilde{\cU}\). \hn \(\hat{H}(v,0)=F''(v)\) and \(\hat{H}(v,1)=\hat{F}(v)\) \fa $v\in \bX$ and \(\hat{H}(v,t)=F''(v)\) \fa \((v,t)\in \bA \times I\). We extend the definitions of \(\hat{H}\) and \(\hat{F}\) on all of \(I^k\times I\) and \(I^k\) by setting \(\hat{H}(v,t)=F''(v)\) and \(\hat{F}(v)=F''(v)\) \fa \(v\in \ov{I^k(l)\setminus \bX}\) and \(t\in I\). Then \eqref{e.A} implies \(\hat{H}:(I^k,\del I^k)\times I \ra (\La M, \La (M\setminus U))\) and \(\hat{F}:I^k \ra \La (M\setminus U)\) are well-defined, \cts maps with \(\hat{H}(v,0)=F''(v)\) and \(\hat{H}(v,1)=\hat{F}(v)\) \fa \(v\in I^k\). Hence \([F'']=0\in \pi_k(\La M,\La (M\setminus U))\), which contradicts \eqref{e.[F''].nonzero}. This finishes the proof of Claim \ref{c.geodesic}.
\end{proof}

\begin{proof}[Proof of Theorem \ref{thm} when either (C2) or (C3) holds]
If (C2) holds, \te \(\Pi_0,\;\Pi_1\in \pi_0(\La(M\setminus U))\) and \(\tilde{\Pi}\in \pi_0(\La M)\) \st \(\Pi_0,\;\Pi_1\subset \tilde{\Pi}\). Since \(M\setminus U\) is connected, without loss of generality we can assume that \(\Pi_0\) does not contain the constant loops in \(\msu\). Let \(P:I\ra \tilde{\Pi} \) be a \cts map \st \(P(0)\in \Pi_0\) and \(P(1)\in \Pi_1\). Then \([P]\neq 0\in \pi_1(\La M, \La (\msu),c^*)\), where \(c^*\in \Pi_0\). On the other hand, if (C3) holds, then the homotopy exact sequence of \((\La M, \La (\msu),c^*)\) implies either \(\pi_m(\La M, \La (\msu),c^*)\neq 0\) or \(\pi_{m+1}(\La M, \La (\msu),c^*)\neq 0\). So let  us assume that \(\pi_k(\La M, \La (\msu),c^*)\neq 0\) for some \(k\geq 1\) and \(c^*\in \La (\msu)\) which is non-contractible in \(\msu\).
	
We now proceed as in the proof of Theorem \ref{thm} when (C1) holds. We choose \(\tilde{F}:(I^k,\del I^k)\ra (\La M, \La (M\setminus U))\) \sot 
\begin{equation}\label{e.[Ft].nonzero}
[\tilde{F}]\neq 0\in \pi_k(\La M, \La (M\setminus U),c^*);
\end{equation}
repeating the argument of Parts 1 -- 7 we construct $F'':(I^k,\del I^k)\ra (\La M, \La (M\setminus U))$ \st for $v\in I^k$, if \(F''(v)\notin \La (\msu)\), then (cf. \eqref{e.F''.non-cnst})
\begin{equation}\label{e.F''2.non-cnst}
F''(v) \text{ is constant.}
\end{equation}
Moreover, \tes \(G: (I^k,\del I^k)\times I\ra (\La M, \La (M\setminus U))\) \w \(G(-,0)=\tilde{F}\) and \(G(-,1)=F''\). \tf if \(\Pi\) denotes the path-component of \(\La(\msu)\) which contains \(c^*\), then \((F'')^{-1}(\Pi)\neq I^k\) (by \eqref{e.[Ft].nonzero}); moreover, \((F'')^{-1}(\Pi)\) is a closed subset of \(I^k\). \tf if $u$ belongs to the (topological) boundary of \((F'')^{-1}(\Pi)\), then \eqref{e.F''2.non-cnst} implies \(F''(u)\) is constant. This contradicts the fact that \(c^*\) is non-contractible in \(\msu\).

\end{proof}

\begin{proof}[Proof of Theorem \ref{thm} when (C4) holds] The proof is divided into five parts.
	
	\noindent\textbf{Part 1.} Condition (C4) implies that \tes \(\tilde{\ga}\in \La M\) \st 
	\begin{equation}\label{e.[ga].nonzero}
	\tilde{\ga} \text{ is not freely homotopic to any loop } \hat{\ga} \in \La (M\setminus U) \text{ on } M.
	\end{equation}
	Let us choose \(K\in \bbn\) \sot \(\md{s-s'}\leq 1/K\) implies
	\[d(\tilde{\ga}(s),\tilde{\ga}(s'))<\frac{1}{2}\inf\{\text{conv}_p(M):p\in \md{\tilde{\ga}}\}.\]
	Let \(\ga:S^1\ra M\) be defined by the condition that \fa \(i\in \{0,1,2,\dots K\}\), \(\ga(i/K)=\tilde{\ga}(i/K)\) and \fa $i\in [K]$, \(\ga\) restricted to \([\frac{i-1}{K},\frac{i}{K}]\) is a minimizing constant speed geodesic. Then $\ga\in \sL M$. Let us define the homotopy \(J_0:S^1\times [0,1]\ra M\) by \(J_0(s,t)=\bH(\tilde{\ga}(s),\ga(s))(t)\), where \(\bH\) is as in \eqref{e.bH}. Then \(J_0(-,0)=\tilde{\ga}\) and \(J_0(-,1)=\ga\).
	
	\noindent\textbf{Part 2.} Let \(\ell(\ga)=\la\), \(V=\cN_{\la}(U)\),
	\[\al=\inf \{\text{conv}_p(M):p\in V \}, \]
	\(W=\cN_{2\al}(V)\). Let us choose a closed \Rm \mf \(N\) \st \(N\) and \(M\) have the same dimension and \tes an isometric embedding \(\iota:W \ra N\). We will abuse notation by identifying \(W\) \w \(\iota(W)\).  Let \(\sE(\ga)=E\). We choose \(R>0\) and \(L\in \bbn\) \sot \(4R<\al\) and \((E,R,L)\) satisfies \eqref{e.R}. For this fixed choice of \((E,R,L)\), let us consider the maps \(\Ps\) and \(\Ph\) as defined in Definition \ref{d.Ps} and Proposition \ref{p.Ph}.
	
	\noindent\textbf{Part 3.} Let $\ga_0=\ga$ and for \(i\in\bbn\), let \(\ga_i=\Ps^i(\ga)\). 
	By Proposition \ref{p.Ph} and Remark \ref{r.spt.Ps.Ph}, \fa \(i\in\bbn\), \tes \(J_i:S^1\times[0,1]\ra N\) \st 
	\begin{equation}\label{e.H_i}
	J_i(-,0)=\ga_{i-1},\; J_i(-,1)=\ga_i\text{ and } \md{J_i(-,t)}\subset \cN_{4R}(\md{\ga_i})\;\forall\;t\in[0,1].
	\end{equation}
	For \(j\in\bbn\), concatenating all the \(J_i\)'s for $1\leq i\leq j$, we obtain \(\tilde{J}_j:S^1\times [0,1]\ra N\) \st \(\tilde{J}_j(-,0)=\ga\) and \(\tilde{J}_j(-,1)=\ga_j\). 
	
	\noindent\textbf{Part 4.} We claim that \(\md{\ga_i}\cap U\neq \emptyset\) \fa \(i\). To justify this claim, let us assume that \tes \(m\in \bbn\) \st \(\md{\ga_i}\cap U\neq \emptyset\) for \(0\leq i\leq m-1\) and \(\md{\ga_m}\cap U=\emptyset\). By Proposition \ref{p.len.Ps.ga}, 
	\begin{equation}\label{e.ell.ga_i}
	\ell(\ga_i)\leq \la\; \forall\;i\in\bbn.
	\end{equation} 
	\hn \(\md{\ga_i}\subset \cN_{\la}(U)\) \fa \(0\leq i\leq m-1.\) \tf by \eqref{e.H_i}, \(\md{J_i(-,t)}\subset \cN_{\al}(V)\) \fa \(i\in[m]\) and $t\in [0,1]$. Thus, \(\tilde{J}_m:S^1\times [0,1]\ra W\) \w \(\tilde{J}_m(-,0)=\ga\) and \(\tilde{J}_m(-,1)=\ga_m\). \tf viewing \(W\) as a subset of \(M\), $\tilde{\ga}$ is freely homotopic to $\ga_m\in \La (M\setminus U)$ on $M$, which contradicts \eqref{e.[ga].nonzero}.
	
	\noindent\textbf{Part 5.} By Proposition \ref{p.len.Ps.ga}, \tes a subsequence \(\{\ga_{i_k}\}_{k=1}^{\infty}\) and a closed geodesic \(\si\) in \(N\) \st \(\ga_{i_k}\) converges to \(\si\) in \(\sL N\). By \eqref{e.ell.ga_i}, \(\ell(\si)\leq \la\). We claim that \(\ell(\si)>0\) and \(\md{\si}\cap U\neq \emptyset\). Assuming the claim, we obtain \(\md{\si}\subset \cN_{\la}(U)\). Viewing \(W\) as a subset of \(M\), we conclude that \(\si\) is a non-trivial closed geodesic in \(M\).
	
	To justify the claim, we first note that convergence in \(\sL N\) implies convergence in \(\La N\). \tf \tes \(l\in \bbn\) \st \(d(\si(s),\ga_{i_l}(s)) < \al\) \fa \(s\in S^1\). Moreover, by Part 4, \(\md{\si}\cap \overline{U}\neq \emptyset\); so \(\md{\si}\subset \cN_{\la}(U)\). We define the homotopy \(\tilde{J}:S^1\times [0,1]\ra W\) by \[\tilde{J}(s,t)=\bH(\si(s),\ga_{i_l}(s))(t),\] where \(\bH\) is as in \eqref{e.bH}. Further, by \eqref{e.H_i} and Part 4, \(\text{image}(\tilde{J}_j)\subset W\), \fa \(j\in \bbn\). Thus, concatenating \(J_0,\;\tilde{J}_{i_l}\) and \(\tilde{J}\), we obtain \(J:S^1\times [0,1]\ra W\) \st \(J(-,0)=\tilde{\ga}\) and \(J(-,1)=\si\). \tf viewing \(W\) as a subset of \(M\), $\tilde{\ga}$ is freely homotopic to $\si$ on $M$. \hn if \(\md{\si}\) is a point or \(\md{\si}\cap U=\emptyset\), then that will be a contradiction of \eqref{e.[ga].nonzero}. 
	
\end{proof}

\begin{lem}\label{l.conj}
Let \(G\) be an arbitrary group and \(H\) be a proper finite index or normal subgroup of \(G\). Then
\begin{equation}\label{e.conj}
G\neq \bigcup_{g\in G}gHg^{-1}.
\end{equation}
 
\end{lem}
 \begin{proof}
If $H$ is a proper normal subgroup of \(G\) then \(gHg^{-1}=H\) \fa \(g\in G\); hence \eqref{e.conj} holds.
 
If $H$ is a proper finite index subgroup of $G$, we consider the set $S$ of the left cosets of $H$ in $G$. Let $\mathfrak{S}$ denote the permutation group of $S$ and \(\pi:G\ra \mathfrak{S}\) be the homomorphism \sot \(\pi(x)\in \mathfrak{S}\) maps \(gH\) to \(xgH\). Let $G'=\pi(G)$. By Burnside's lemma,
\begin{equation}\label{e.BL}
1=\frac{1}{\md{G'}}\sum_{\ta\in G'}\md{\text{Fix}(\ta)},
\end{equation}
where for $\ta\in \mathfrak{S}$, \(\text{Fix}(\ta)\) is the set of the fixed points of $\ta$ in $S$. Since the identity element of \(G'\) has more than one fixed points, \eqref{e.BL} implies \tes \(\ta_0\in G'\) \st \(\text{Fix}(\ta_0)=\emptyset\). \(x\in gHg^{-1}\) (where $x,\; g\in G$) if and only if \(xgH=gH\). \tf if \(\ta_0=\pi(x_0)\), \(x_0\notin  \bigcup\limits_{g\in G}gHg^{-1}.\)
\end{proof}

\begin{rmk}\label{r.conj}
Let \(\mathbb{F}\) be an algebraically closed field. Let \(G\) be the group of all \(2\times 2\) invertible matrices whose entries are in \(\mathbb{F}\) and \(H\) be the group of all \(2\times 2\) invertible upper triangular matrices whose entries are in \(\mathbb{F}\). Then, there holds \(G= \bigcup\limits_{g\in G}gHg^{-1}.\)
\end{rmk}

\begin{lem}\label{l.*}
Let \(X\) be a connected topological manifold. Suppose \(A\) is a connected open subset of \(X\) \sot \(\overline{A}\) is a compact manifold \w connected boundary \(\del A\). Let us consider the following three conditions.
\begin{itemize}
	\item[(i)] The map \(\pi_0(\La (X\setminus A))\ra \pi_0(\La X)\) (induced by the inclusion $X\setminus A \hookrightarrow X$) is surjective.
	\item[(ii)] The union of all the conjugate subgroups of the image of the homomorphism \(\pi_1(X\setminus A)\ra \pi_1(X)\) (induced by the inclusion $X\setminus A \hookrightarrow X$) is \(\pi_1(X)\).
	\item[(iii)] The union of all the conjugate subgroups of the image of the homomorphism \(\pi_1(\del A)\ra \pi_1(\overline{A})\) (induced by the inclusion $\del A \hookrightarrow \overline{A}$) is \(\pi_1(\overline{A})\). 
\end{itemize}

Then (i) and (ii) are equivalent and they imply (iii).
\end{lem}
\begin{proof}
Given a group \(G\), let \(\text{conj}(G)\) denote the set of conjugacy classes in \(G\). A group homomorphism \(\ph:H\ra G\) induces a map \(\ph_{\#}:\text{conj}(H)\ra \text{conj}(G)\) which maps the conjugacy class of \(h\) to the conjugacy class of \(\ph(h)\). $\ph_{\#}$ is surjective if and only if 
\begin{equation*}
G=\bigcup_{g\in G}g\ph(H)g^{-1}.
\end{equation*}
Thus, the homomorphism \(\pi_1(X\setminus A)\ra \pi_1(X)\) (induced by the inclusion $X\setminus A \hookrightarrow X$) induces a map \(\text{conj}(\pi_1(X\setminus A))\ra \text{conj}(\pi_1(X))\). For a topological space $Y$, \(\pi_0(\La Y)\) can be canonically identified with \(\text{conj}(\pi_1(Y))\). Under this identification, the map \(\pi_0(\La (X\setminus A))\ra \pi_0(\La X)\) (induced by the inclusion $X\setminus A \hookrightarrow X$) corresponds to the map \(\text{conj}(\pi_1(X\setminus A))\ra \text{conj}(\pi_1(X))\), described above. \tf (i) and (ii) are equivalent.

By the Seifert-van Kampen theorem, $\pi_1(X)$ is isomorphic to the amalgamated free product $\pi_1(\ov{A})\ast_{\pi_1(\del A)}\pi_1(X\setminus A)$. From the properties of the amalgamated free product of groups \cite{MO1,MO2}, \cite[Theorem 4.6(ii)]{MKS} it follows that (ii) implies (iii).
\end{proof}

\begin{rmk}\label{r.comp}
Let \(X\) be a topological space and \(A\subset X\) be connected. Let us compare the three conditions $\pi_1(X,A)\neq 0$, \(\pi_0(\La A)\ra \pi_0(\La X)\) is not surjective and \(H_1(X,A)\neq 0\), which appear in Theorems \ref{thm} and \ref{thm.main}, \w $(X,A)=(M,\msu)$. From Lemma \ref{l.conj}, Remark \ref{r.conj}, the existence of finite perfect groups and the following discussion, it follows that, in general, these three conditions are not equivalent. The homotopy exact sequence for the pair $(X,A)$ implies \(\pi_1(X,A)\neq 0\) if and only if $\pi_1(A)\rightarrow \pi_1(X)$ is not surjective. On the other hand, in the proof Lemma \ref{l.*}, we explained that \(\pi_0(\La A)\ra \pi_0(\La X)\) is not surjective if and only if \(\text{conj}(\pi_1(A))\ra \text{conj}(\pi_1(X))\) is not surjective. Lastly, the homology exact sequence for the pair $(X,A)$ implies that \(H_1(X,A)\neq 0\) if and only if $\pi_1(A)^{ab}\rightarrow \pi_1(X)^{ab}$ is not surjective (where $G^{ab}$ denotes the abelianization of the group $G$).
\end{rmk}

\begin{proof}[Proof of Theorem \ref{thm.main}] As mentioned in Remark \ref{r.1.4}, by Theorem \ref{thm}, \tes a closed geodesic on $M$ if either \(\pi_1(M,\msu)=0\) or the map \(\pi_0(\La (\msu))\ra \pi_0(\La M)\) is not surjective. Moreover, by the homotopy exact sequence for the pair $(M,\msu)$, \(\pi_1(M,\msu)=0\) if and only if $\pi_1(M\setminus U)\rightarrow \pi_1(M)$ is surjective. By the Seifert-van Kampen theorem, if $\del U$ is connected and $\pi_1(\del U)\ra \pi_1(\ov{U})$ is surjective, then $\pi_1(M\setminus U)\rightarrow \pi_1(M)$ is also surjective. 
	
\Sn $\msu$ is connected, every relative homology class in \(H_1(M,\msu)\) can be represented by a closed loop \(\ga:[0,1]\ra M\) \st \(\ga(0)=\ga(1)\in \msu\). If \([\ga]\neq 0\in H_1(M,\msu)\), then \(\ga\) can not be freely homotopic to a loop in \(\msu\). \tf $H_1(M,\msu)\neq 0$ implies \(\pi_0(\La (\msu))\ra \pi_0(\La M)\) is not surjective.

Let $\Si_1,\Si_2$ be two connected components of $\del U$. For $i=1,2$, let $p_i\in \Si_i$, $\ga_1:[0,1]\ra \ov{U}$ and $\ga_2:[0,1]\ra \msu$ be $C^1$ embedded curves, $\ga_i^{-1}(\del U)=\{0,1\}$, $\ga_i(0)=p_1$, $\ga_i(1)=p_2$, $\ga_i$ intersects $\del U$ orthogonally. Let $\ga$ be the closed curve obtained by concatenating $\ga_1$ and $\ga_2$. Then the mod $2$ intersection number of $\ga$ with $\Si_1$ is $1$. So, \(\ga\) can not be freely homotopic to a loop in \(\msu\). \tf if $\del U$ is not connected, \(\pi_0(\La (\msu))\ra \pi_0(\La M)\) is not surjective.  

Suppose $\del U$ is connected and $\pi_1(\del U)=0$. If $\pi_1(\ov{U})=0$, then $\pi_1(\del U)\ra \pi_1(\ov{U})$ is surjective which implies \(\pi_1(M,\msu)=0\). If $\pi_1(\ov{U})\neq 0$, Lemma \ref{l.*} implies that \(\pi_0(\La (\msu))\ra \pi_0(\La M)\) is not surjective.

\sps $\pi_1(M)$ is a finite group or an Abelian group. If $\pi_1(M\setminus U)\ra \pi_1(M)$ is surjective, then \(\pi_1(M,\msu)=0\). If $\pi_1(M\setminus U)\ra \pi_1(M)$ is not surjective, Lemmas \ref{l.conj} and \ref{l.*} imply that \(\pi_0(\La (\msu))\ra \pi_0(\La M)\) is not surjective.

\sps $\pi_1(U)$ is a finite group or an Abelian group. We can also assume that $\del U$ is connected. If $\pi_1(\del U)\ra \pi_1(\ov{U})$ is surjective, then \(\pi_1(M,\msu)=0\). If $\pi_1(\del U)\ra \pi_1(\ov{U})$ is not surjective, Lemmas \ref{l.conj} and \ref{l.*} imply that \(\pi_0(\La (\msu))\ra \pi_0(\La M)\) is not surjective.
\end{proof}
\medskip

\bibliographystyle{amsalpha}
\bibliography{geodesic_bib}

\end{document}